\documentclass{article}
\usepackage{amsmath}
\usepackage{color}
\usepackage{amssymb,amsfonts}
\usepackage{graphicx}

\usepackage[a4paper,left=3cm,right=3cm]{geometry}
\newtheorem{theorem}{Theorem}

\newtheorem{assumption}[theorem]{Assumption}

\newtheorem{definition}[theorem]{Definition}
\newtheorem{example}[theorem]{Example}

\newtheorem{lemma}[theorem]{Lemma}

\newtheorem{proposition}[theorem]{Proposition}
\newtheorem{remark}[theorem]{Remark}

\newenvironment{proof}[1][Proof]{\textbf{#1.} }{\ \rule{0.5em}{0.5em}}

\DeclareMathOperator*{\esssup}{ess\,sup}
\DeclareMathOperator*{\essmax}{ess\,max}

\def\bP{P}
\def\cF{\mathcal{F}}
\def\bR{\mathbb{R}}
\def\cE{\mathcal{E}}

\begin{document}

\author{Elisa Mastrogiacomo\footnote{Dipartimento di Economia, Universit\`{a} dell'Insubria, via Monte Generoso 71,
21100 Varese, Italy. elisa.mastrogiacomo@uninsubria.it} \, and Emanuela Rosazza Gianin\footnote{Department of Statistics and Quantitative methods,
University of Milano Bicocca
Via Bicocca degli Arcimboldi 8, 20126 Milano, Italy.
emanuela.rosazza1@unimib.it}}
\title{Dynamic capital allocation rules via BSDEs: an axiomatic approach}

\maketitle

\abstract{
%
In this paper, we study capital allocation for dynamic risk measures, with an axiomatic approach but also by exploiting the relation between risk measures and BSDEs. Although there is a wide literature on capital allocation rules in a static setting and on dynamic risk measures, only a few recent papers on capital allocation work in a dynamic setting and, moreover, those papers mainly focus on the gradient approach.
To fill this gap, we then discuss new perspectives to the capital allocation problem going beyond those already existing in the literature. In particular, we introduce and investigate a general axiomatic approach to
dynamic capital allocations as well as an approach suitable for
risk measures induced by $g$-expectations under weaker assumptions than Gateaux differentiability.
\normalcolor}
\medskip


\section{Introduction}
A relevant research stream related to risk measures is the capital allocation problem, dealing with the problem of sharing in a suitable way the margin required to hedge the riskiness of a position among the different sources of riskiness of the aggregate position. For static risk measures the capital allocation problem has been faced in an axiomatic way but also from an empirical or a game-theoretical point of view (see~\cite{delbaen-pisa,denault,kalkbrener,tasche,tsanakas,canna-etal,centrone-rg} and the references therein).
Although there is a wide literature on the relation between dynamic risk measures and Backward Stochastic Differential Equations (BSDEs, for short) and on capital allocation rules (CARs) in a static setting, a systematic analysis of the dynamic setting has not yet been extensively developed and only a few papers (see~\cite{boonen,cherny,kromer-overbeck-bsde,kromer-overbeck-volterra,mabitsela-etal,tsanakas-dynamic}) on capital allocation work in a dynamic setting. Even if the idea of using BSDEs (or Backward Stochastic Volterra Integral Equations) for dynamic CARs can be found already in~\cite{kromer-overbeck-bsde,kromer-overbeck-volterra,mabitsela-etal}, in these papers very specific CARs are considered and the authors mainly investigate and cover the gradient approach.
\smallskip

Motivated by the relevance of dynamic risk measures and the need of a complete study on CARs in a dynamic setting, in this paper we introduce an axiomatic approach to dynamic capital allocation as well as an approach suitable for risk measures induced by BSDES or, better, by $g$-expectations (see~\cite{peng-97}).
The main goal of the paper is, therefore, to investigate from an axiomatic point of view capital allocation rules of convex risk measures in a dynamic setting by weakening the Gateaux differentiability condition of the risk measure and going beyond the gradient approach. In particular, we generalize the axioms introduced by~\cite{kalkbrener} for capital allocations in a dynamic setting and investigate their relation and compatibility with time-consistency of dynamic risk measures.
Furthermore, when the underlying risk measure is induced by a BSDE governed by a non-smooth (but convex) driver, we introduce dynamic capital allocation rules going beyond the gradient allocation (already represented in terms of the driver in~\cite{kromer-overbeck-bsde,kromer-overbeck-volterra,mabitsela-etal}).
\smallskip

The paper is organized as follows: in the next section we introduce notations and review the main notions and results useful in the paper; in Section~\ref{sec:axioms} we introduce an axiomatic approach to capital allocation rules in a dynamic setting and investigate existence of rules fulfilling some further axioms. In Section~\ref{sec:g-expectation}, instead, we focus on dynamic capital allocation rules induced by $g$-expectations and present some examples. Conclusions and final remarks are provided in Section~\ref{sec: conclus}.


\section{Notations and preliminaries}\label{sec:preliminaries}
In this section we recall some basic notions and definitions that will be used in the following sections.
We first concern with the notion of dynamic risk measures and the related properties and then move to basic facts about dynamic measures and their connection to BSDEs. \bigskip

\textit{Dynamic risk measures.}\smallskip

Let $T>0$ be a given future time horizon and let $(\Omega, \mathcal{F}, P)$ be a general probability space. Consider a continuous-time setting where time evolves between $0$ and $T$ and let $(\mathcal{F}_t)_{0 \leq t \leq T}$ be a filtration such that $\mathcal{F}_0=\{\emptyset, \Omega\}$ and $\mathcal{F}_T=\mathcal{F}$. We will focus on risk measures quantifying the riskiness of financial positions belonging to $L^{\infty}(\mathcal{F}_T)=L^{\infty}(\Omega, \mathcal{F}_T, P)$, that is, the space of essentially bounded random variables defined on $(\Omega, \mathcal{F}_T, P)$. All equalities and inequalities have to be understood in the $P$-a.s. sense.\smallskip

We recall that a static risk measure is a functional quantifying \textit{now} the riskiness of any position $X$ of
maturity $T$, while a dynamic risk measure is a functional quantifying the riskiness of $X$ at any time $t \in [0,T]$, taking into account the whole information available up to time $t$. More precisely, we recall the following definition by referring to~\cite{ArDeEbHe99,bion-nadal-FS,CDK1,delb,delbaen-pisa,delb-mstable,detlef-scandolo,follmers-book,frittelli-rg}, among many others, for a more detailed treatment on static and dynamic convex risk measures.

\begin{definition}
A static risk measure is a functional $\rho: L^{\infty} (\mathcal{F}_{T}) \to \Bbb R$ satisfying some further assumptions (e.g., monotonicity, cash-invariance, convexity).

A dynamic risk measure $(\rho_{t}) _{t \in [0,T]}$ is a family of functionals
%
%
%
%
%
\begin{equation*}
\rho_{t}: L^{\infty} (\mathcal{F}_{T}) \to L^{\infty}( \mathcal{F}_{t}), \quad t \in [0,T],
\end{equation*}
such that $\rho_0$ is a static risk measure and $\rho_T(X)=-X$ for any $X\in L^\infty(\cF_T)$.
\end{definition}
\normalcolor
\medskip

An incomplete list of desirable properties that are sometimes imposed to dynamic risk measures $(\rho_t)_{t\in [0,T]}$ are the following: \medskip

\noindent - monotonicity: if $X,Y \in L^{\infty}(\mathcal{F}_{T})$ and $X \leq Y$, then $\rho_{t}(X) \geq \rho_{t}(Y)$ for any $t \in [0,T]$.
\smallskip

\noindent - cash-additivity:
$\rho_{t}(X+m_t)=\rho_{t}(X)-m_t$ for any
$X \in L^{\infty} (\mathcal{F}_{T})$, $m_t \in L^{\infty}(\mathcal{F}_t)$ and $t \in [0,T]$.

\noindent - convexity: $\rho_{t} (\alpha X +(1-\alpha)Y) \leq \alpha \rho_{t}(X)+ (1-\alpha)
\rho_{t}(Y)$ for any $X,Y \in L^{\infty}
(\mathcal{F}_{T})$, $t \in [0,T]$ and $\alpha \in [0,1]$.
\smallskip


\noindent - normalization: $\rho_t(0) =0$,  for any $t\in [0,T]$.\smallskip

\noindent - time-consistency:
$\rho_s(X)= \rho_s(-\rho_t(X))$, for any $X\in L^\infty(\cF_T)$, $0\leq s \leq t \leq T$.
\smallskip

\noindent - weak time-consistency:
$\rho_s(X)\leq  \rho_s(-\rho_t(X))$, for any $X\in L^\infty(\cF_T)$, $0\leq s \leq t \leq T$.
\normalcolor
\medskip

While monotonicity, cash-additivity, convexity and normalization are essentially a translation to the dynamic setting of the corresponding static axioms, time-consistency is peculiar to the dynamic setting and is a recursivity property. Weak time-consistency, instead, is a weaker version of recursivity.

In the following, with dynamic convex risk measures we mean any dynamic risk measure satisfying monotonicity, cash-additivity, convexity, and normalization.
We recall the following results on the dual representations of risk measures (see~\cite{bion-nadal-SPA} and~\cite{detlef-scandolo}).

Let $(\rho_{t}) _{0 \leq t \leq T}$
be a dynamic convex risk measure satisfying continuity from below, that is, if
$(X_n)_{n \geq 0} \subseteq L^{\infty}(\mathcal{F}_T)$ is an increasing sequence with $X_n \nearrow X$, then
$\rho_t(X_n) \to_{n \to + \infty} \rho_t(X)$, $P$-a.s.
Then it can be represented as
\begin{equation} \label{eq: repr dyn convex addit}
\rho_{t}(X)=
\essmax_{Q \in \mathcal{Q}_{t}} \{ E_Q [- X \vert \mathcal{F}_ {t} ]
-c_{t}(Q) \},
\end{equation}
where $c_{t}$ is the so called
(minimal) penalty term, defined as
\begin{equation} \label{eq: c penalty}
c_t(Q)\triangleq \esssup_{Y \in L^\infty(\mathcal{F}_T)} \{E_{Q} [-Y \vert \mathcal{F}_t] -\rho_t(Y) \}, \quad Q\in \mathcal{P}_t,
\end{equation}
with $\mathcal{P}_{t}$ being the set of probability measures defined on $(\Omega,\mathcal{F}_T)$, absolutely continuous with respect to $P$ and such that $Q \vert _{\mathcal{F}_t}=P$, and $\mathcal{Q}_t \triangleq \{Q \in \mathcal{P}_t: c_t(Q) \in L^{\infty}(\mathcal{F}_t)\}$.
\normalcolor
Coherent risk measures correspond to penalty terms belonging to $\{0; + \infty\}$ (see~\cite{delb-mstable}).
\bigskip

\textit{BSDEs and their connection to dynamic risk measures.}\smallskip

Let us now recall the connection between dynamic risk measures and BSDEs. Let $(B_t)_{t \geq 0}$ be a standard $d$-dimensional Brownian motion defined on the probability space
$(\Omega,\mathcal{F},P)$. Denote by
$(\mathcal{F}_t)_{t\geq 0}$ the natural Brownian filtration augmented by the $P$-null sets.
In the sequel, we identify a probability measure $Q< \!\!\!< P$ with its Radon-Nikod\'ym density
$\frac{{\rm d }Q}{{\rm d} P}$. Since we are working in a Brownian setting, we also identify a probability measure $Q$ equivalent to $P$ with the predictable process $(q_t)_{t\in [0,T]}$ induced by the stochastic exponential, i.e., such that
\begin{align*}
E\left[\left. \frac{d Q}{d P}\right|\cF_t\right] = \cE(q . B)_t
= \exp\left( -\frac{1}{2} \int_0^t \|q_s\|^2 ds + \int_0^t q_s d B_s\right).
\end{align*}
 See Revuz and Yor~\cite{revuz-yor} for details. \normalcolor
Moreover, given a probability measure $Q$ on  $(\Omega,\mathcal{F})$, we denote by $E_Q$ (respectively,
$E_Q[\ \cdot \ | \cF_t]$) the expected value (respectively, the conditional expectation) with respect to
$Q$. When the expectation is considered with respect to the reference probability measure $P$, we will simply use the notation $E$.

Pardoux and Peng~\cite{pardoux-peng90} introduced equations of the following type (known as BSDEs):
\begin{align}\label{eq:BSDE}
Y_t=X+ \int_t^T g(s,Y_s,Z_s) ds -\int_t^T Z_s d B_t, \quad 0\leq t\leq T
\end{align}
where $g:\Omega \times [0,T]\times \bR \times \bR^{d} \to \bR$ is generally called generator or driver, $T$ is the time horizon and $X \in L^2(\cF_T)$
is a terminal condition, where $L^2(\mathcal{F}_T)$ is the space of all square integrable random variables on $(\Omega, \mathcal{F}_T, P)$.
To simplify the notation we often write $g(t,y,z)$ instead of
$g(\omega,t,y,z)$.

We know from~\cite[Theorem 4.1]{pardoux-peng90} and~\cite[Theorem 35.1]{peng-97}  or~\cite{EPQ} that if $g$ satisfies the {\it usual assumptions}:
\smallskip

%

\noindent   ($g$1) \  $g$ Lipschitz continuous in $y$ and $z$, i.e., there exists a constant $\mu >0$ such that $d\bP \times d t$, for any
$(y_0,z_0),(y_1,z_1)\in \bR\times \bR^d$,
$$
  |g(t,y_0,z_0) -g(t,y_1,z_1)| \leq \mu \left( |y_0-y_1|+ \| z_0-z_1 \| \right);
$$

\smallskip

\noindent  ($g$2) \  for any $(y,z)\in \bR\times \bR^d$, $g(\cdot,y,z)$ is a predictable process such that
\begin{align}\label{g-condition}
   E\left[ \int_0^T |g(s,y,z)|^2 \, ds\right]<+\infty;
\end{align}
\smallskip

\noindent  ($g$3) \  $g(t,y,0)=0$ for any $t\in [0,T]$, 

\medskip

then, for every
$X \in L^2(\cF_T)$, the BSDE in~\eqref{eq:BSDE} admits a unique solution $(Y_t^X,Z_t^X)$ consisting of predictable processes with values in $\bR \times \bR^d$ such that
$$
 E\left[ \int_0^T Y_t^2 dt \right] <+\infty \quad \text{and} \quad
 E\left[ \int_0^T \| Z_t\|^2 dt \right] <+\infty.
$$

Note that the existence and uniqueness of the solution is guaranteed also when ($g$3) is replaced by \smallskip

\noindent  ($g$3') \   $E\left[ \int_0^T |g(s,y,0)|^2 ds\right]<+\infty$ for any $y\in \bR$.
\smallskip

\noindent In such a case, the driver will be said to satisfy the {\it non-normalized usual assumptions}.\medskip

Using Peng~\cite{peng-97}'s terminology, the first component
of the solution to the BSDE~\eqref{eq:BSDE}
$$
\cE^g(X|\cF_t) \triangleq  Y_t^X
$$
is called conditional $g$-expectation at time $t$. When
$g(t,z)=\mu \|z\|$ (with $\mu >0$), $\cE_g$ will be denoted by $\cE_\mu$ (see, e.g.,~\cite{peng-97}).
Some relevant results about the link between $g$-conditional expectations and dynamic risk measures can be found in~\cite[Section 4]{RG} and~\cite{barrieu-elkaroui}.

As shown in~\cite{RG}, for a convex driver satisfying the usual assumptions ($g$1)-($g$3) and independence from $y$,
$$
\rho_t^g(X) \triangleq \cE^g(-X|\cF_t), \qquad \ X\in L^2(\cF_T),t\in [0,T],
$$
is a dynamic convex and time-consistent risk measure satisfying normalization.

\section{Axiomatic approach to capital allocation: from a static to a dynamic setting}\label{sec:axioms}

In this section, we introduce and investigate an axiomatic approach to dynamic capital allocation in the spirit of~\cite{kalkbrener}.\medskip

We start by recalling some theoretical aspects of capital allocation in a static framework (see~\cite{delbaen-pisa,denault,centrone-rg,kalkbrener,tasche,tsanakas}, among others, for a detailed treatment). Assume that a time horizon $T$ and a static risk measure $\rho$ have been fixed and let $X \in L^{\infty}(\cF_T)$ be a financial position (with maturity $T$) which is formed by the sub-units (or business lines or sub-portfolios) $X_1,\dots,X_n \in L^{\infty}(\cF_T)$, i.e. $X = \sum_{i=1}^n X_i$. For instance, we can think at $X$ as the profit and loss of a portfolio composed by different assets and at $X_1, ...,X_n$ as at the profits and losses of the different assets, or at $X$ as the riskiness of a firm and at $X_1, ...,X_n$ as at the different branches or business lines of the firm.

The {\it capital allocation problem} consists then in finding a ``suitable way'' to share the margin $\rho(X)$ among the different sub-units $X_1,\dots,X_n$. More concretely, it consists in finding $k_1,\dots,k_n \in \bR$ such that
$\rho(X)=k_1+\dots+k_n$ where each $k_i$ denotes the capital to be allocated to $X_i$ or, in other words, the risk contribution of $X_i$ to the total risk capital $\rho(X)$ of $X$.

More in general, given a convex risk measure $\rho$, Kalkbrener~\cite{kalkbrener} defined {\it capital allocation rule (CAR)} with respect to $\rho$ any functional $\Lambda: L^{\infty}(\cF_T) \times L^{\infty}(\cF_T) \to \bR$ satisfying $\Lambda(X;X)=\rho(X)$. In the terminology above, then, $\Lambda(X_i;X)$ represents the capital to be allocated to the sub-unit $X_i$ of the whole position $X$. However, $\Lambda$ is defined for any pair $(X,Y) \in L^{\infty}(\cF_T) \times L^{\infty}(\cF_T)$ where $X$ can be seen as a sub-portfolio\footnote{To be more precise, given a position $Y$, $X$ is called sub-portfolio of $Y$ if $Y=X+Z$ for some $Z$. In particular, in $L^{\infty}(\cF_T)$ any pair $(X,Y)$ can be seen as sub-portfolio and portfolio, respectively. See~\cite{centrone-rg}.} of the whole portfolio $Y$, and $\Lambda(X;Y)$ can be interpreted as the capital allocated to the sub-portfolio $X$ to cover the riskiness of the global portfolio $Y$. Furthermore, the condition $\Lambda(X;X)=\rho(X)$ imposed to a CAR guarantees that for a stand-alone portfolio the capital to be allocated corresponds exactly to the margin (or capital requirement) $\rho(X)$.
Note that
a general CAR may fail to satisfy the requirement $\rho(X)=\sum_{i=1}^n \Lambda(X_i;X)$.
If such a condition is fulfilled,
\normalcolor
then the CAR is said to satisfy full allocation.

Before considering the dynamic version of capital allocation rules, we recall two well-known and quite popular (static) CARs used in the literature: the gradient and the Aumann-Shapley CARs (see~\cite{delbaen-pisa,denault,centrone-rg,kalkbrener}). If $\rho$ is Gateaux differentiable at $Y$, then:

\begin{itemize}
\item the gradient CAR is defined as the directional derivative of $\rho$ at $Y$ in the direction of $X$, i.e.
$$\Lambda^{grad} (X;Y) \triangleq \lim_{h \to 0} \frac{\rho(Y + hX)- \rho(Y)}{h},$$
and can be interpreted as the marginal contribution of sub-unit $X$ to the overall risk $Y$;

\item the Aumann-Shapley CAR is defined as
$$\Lambda^{AS} (X;Y) \triangleq \int_{0}^1 \Lambda^{grad} (X;\gamma Y) d \gamma$$
and somehow corresponds to the Aumann-Shapley value of game theory. We recall, indeed, from Denault~\cite{denault} that the Aumann-Shapley value for player/portfolio $i$ is defined as
$$
\Lambda_i^{AS} \triangleq \int_{0}^1 \frac{\partial r}{\partial \delta_i} (\gamma \Delta) \,d \gamma,
$$
where $\delta \in \Bbb R^n_+$ represents the level of participation of the $n$ players in a coalition, the components of $\Delta \in \Bbb R^n_+$ the full participation of the players, and $r: \Bbb R^n \to \Bbb R$ denotes a cost function. In the context of fractional players, the Aumann-Shapley value can be then interpreted as the average of the marginal costs of player/portfolio $i$, where the average is taken with respect to $\gamma$, acting on the size of portfolio (from $0$ to $\Delta$).
\normalcolor
\end{itemize}

For non-Gateaux differentiable risk measures, subdifferential versions of the previous (static) CARs are defined and studied in~\cite{centrone-rg}.
\bigskip

Although there is a wide literature on dynamic risk measures and on capital allocation rules in a static setting, a systematic analysis of dynamic capital allocations has not yet been extensively developed and only a few papers (see~\cite{boonen,cherny,kromer-overbeck-bsde,kromer-overbeck-volterra,mabitsela-etal,tsanakas-dynamic}) on capital allocation work in a dynamic setting. Furthermore, in these papers very specific CARs are considered and the authors mainly investigate and cover the gradient approach. For this reason, we start providing an axiomatic approach to CARs in the dynamic setting in full generality. In the next section, inspired by \cite{kromer-overbeck-bsde,kromer-overbeck-volterra,mabitsela-etal}, we then introduce an approach suitable for risk measures induced by BSDEs.\medskip

The general notion of a static capital allocation rule can be immediately generalized to a dynamic framework.
\begin{definition}
Given a dynamic risk measure $(\rho_t)_{t \in [0,T]}$, we define dynamic capital allocation rule (CAR) associated to $(\rho_t)_{t \in [0,T]}$ a family $(\Lambda_t)_{t \in [0,T]}$ of functionals
$$\Lambda_t: L^{\infty}(\mathcal{F}_T) \times L^{\infty}(\mathcal{F}_T) \to L^{\infty}(\mathcal{F}_t)$$
satisfying $\Lambda_t(X;X)=\rho_t(X)$ for any $X \in L^{\infty}(\mathcal{F}_T)$ and $t \in [0,T]$.

An audacious dynamic CAR, instead, will only satisfy $\Lambda_t(X;X)\leq \rho_t(X)$ for any $X \in L^{\infty}(\mathcal{F}_T)$ and $t \in [0,T]$ (see~\cite{centrone-rg} for the static version).
\end{definition}

Similarly to the static case, $\Lambda_t(X;Y)$ can be interpreted as the (random) amount to be allocated to $X$ as a sub-portfolio of $Y$ at time $t$. Differently from a static CAR where $\Lambda(X;Y)$ is deterministic, for a dynamic CAR $\Lambda_t(X;Y)$ is a $\cF_t$-measurable random variable, hence taking into account all the information available up to time $t$.\smallskip

Here below we provide a list of desirable axioms which extend to the dynamic setting those defined for static CARs in~\cite{kalkbrener,denault,centrone-rg,tsanakas}. In a static setting, no-undercut and full allocation were already studied in~\cite{kalkbrener} while the others have been introduced and discussed later on.\medskip
\normalcolor

\noindent - monotonicity: if $X \leq Z$ then $\Lambda_t(X;Y) \geq \Lambda_t(Z;Y)$ for any $Y \in L^{\infty}(\mathcal{F}_T)$ and $t \in [0,T]$.
\smallskip

\noindent - no-undercut: $\Lambda_t(X;Y) \leq \rho_t(X)$ for any $X,Y \in L^{\infty}(\mathcal{F}_T)$ and $t \in [0,T]$.
\smallskip

\noindent - $\mathcal{F}_t$-riskless: $\Lambda_t(c_t;Y) =-c_t $ for any $c_t \in L^{\infty}(\mathcal{F}_t), Y \in L^{\infty}(\mathcal{F}_T)$ and $t \in [0,T]$.
\smallskip

\noindent - $\mathcal{F}_t$-1-cash-additivity: $\Lambda_t(X+c_t;Y) =\Lambda_t(X;Y)-c_t $ for any $c_t \in L^{\infty}(\mathcal{F}_t), X,Y \in L^{\infty}(\mathcal{F}_T)$ and $t \in [0,T]$.
\smallskip

\noindent - $\mathcal{F}_t$-cash-additivity: $\Lambda_t(X+c_t;Y+c_t) =\Lambda_t(X;Y)+c_t $ for any $c_t \in L^{\infty}(\mathcal{F}_t), X,Y \in L^{\infty}(\mathcal{F}_T)$ and $t \in [0,T]$.
 \smallskip

\noindent - full allocation: $\Lambda_t(\sum_{i=1}^n Y_i;Y) = \sum_{i=1}^n \Lambda_t(Y_i;Y)$ for any $Y_1,...,Y_n, Y \in L^{\infty}(\mathcal{F}_T)$ such that $Y=\sum_{i=1}^n Y_i$, and $t \in [0,T]$.
\smallskip

\noindent - sub-allocation: $\Lambda_t(\sum_{i=1}^n Y_i;Y) \geq  \sum_{i=1}^n \Lambda_t(Y_i;Y)$ for any $Y_1,...,Y_n, Y \in L^{\infty}(\mathcal{F}_T)$ such that $Y=\sum_{i=1}^n Y_i$, and $t \in [0,T]$.
\smallskip

\noindent - weak convexity: $\Lambda_t(\sum_{i=1}^n \alpha_i Y_i;Y) \leq  \sum_{i=1}^n \alpha_i \Lambda_t(Y_i;Y)$ for any $\alpha_i \in [0,1]$ ($i=1,...,n$), $Y_1,...,Y_n, Y \in L^{\infty}(\mathcal{F}_T)$ satisfying $Y=\sum_{i=1}^n \alpha_i Y_i$ and $\sum_{i=1}^n \alpha_i=1$, and $t \in [0,T]$.
\medskip

Monotonicity means that the capital allocated to a position $Z$ has to be smaller than the capital allocated to a riskier position $X$. No-undercut translates the idea that, at any time $t$, the capital allocated to any $X$ considered as a sub-portfolio of $Y$ does not exceed the capital allocated to $X$ considered as a stand-alone portfolio. In the terminology of Tsanakas~\cite{tsanakas}, such property guarantees that there is no incentive to split $X$ from $Y$ because the capital requirement due to $X$ as a stand-alone portfolio would be higher than the capital to be allocated to $X$ as a sub-portfolio of $T$.
$\mathcal{F}_t$-riskless means that the capital allocated to a, roughly speaking, ``known'' position at time $t$ (or, better, to a $\cF_t$-measurable r.v.) is exactly the opposite of such position.
$\mathcal{F}_t$-1-cash-additivity and $\mathcal{F}_t$-cash-additivity have a similar interpretation to cash-additivity for dynamic risk measures; we stress that for $\mathcal{F}_t$-1-cash-additivity the translation has an impact only on the first variable, i.e., on the sub-portfolio. 
Full allocation states that, at any time $t$, the capital requirement $\Lambda_t(Y;Y)$ is fully divided into the different sub-portfolios. As emphasized by Kalkbrener~\cite{kalkbrener} in the static case, however, full allocation and no-undercut together with $\Lambda(X;X)=\rho(X)$ are incompatible for convex risk measures that are not coherent since these axioms together imply subadditivity. Nevertheless, as underlined by Brunnermeier and Cheridito~\cite{brunnemeier-cheridito}, full allocation can be dropped when, e.g., the CAR is considered only for monitoring purposes (see also~\cite{canna-etal,centrone-rg} and the references therein for a deeper discussion). For the reason above, sub-allocation and weak convexity can be defined and investigated as alternatives to full allocation.
In particular, sub-allocation implies that the excess $\Lambda_t(\sum_{i=1}^n Y_i;Y) -  \sum_{i=1}^n \Lambda_t(Y_i;Y) \geq 0$ can be seen as an undivided deposit/cost that represents an extra-security margin and can be motivated because of some costs shared by all the sub-portfolios. Weak convexity, instead, represents a sort of convexity in the first variable.
\medskip

While the previous axioms are simply a reformulation in a dynamic setting of those required in the static case, the following are specific for dynamic CARs and, up to our knowledge, have not been introduced yet in the literature of CARs.
In particular, time-consistency is specific for the dynamic setting and represents a recursivity requirement, similarly to the one imposed to dynamic risk measures.
\normalcolor

\noindent - time-consistency of type 1: $\Lambda_s(-\Lambda_t(X;Y);Y) =\Lambda_s(X;Y)$ for any $X,Y \in L^{\infty}(\mathcal{F}_T)$ and $0 \leq s\leq t \leq T$.\smallskip

\noindent - time-consistency of type 2: $\Lambda_s(-\Lambda_t(X;Y);-\rho_t(Y)) =\Lambda_s(X;Y)$ for any $X,Y \in L^{\infty}(\mathcal{F}_T)$ and $0 \leq s\leq t \leq T$.\smallskip
\medskip

More specifically, time-consistency of type 2 means that the capital to be allocated at time $s$ for $X$ as a sub-portfolio of $Y$ should be the same as the capital to be allocated at time $s$ for the sub-position $-\Lambda_t(X;Y)$ of the whole position $-\rho_t(Y)$, arising from a intermediate step from $T$ to $t$; that is, by considering two consecutive steps from $T$ to $t$ and then from $t$ to $s$, where in the last step we don't consider the final position $Y$ but its margin at time $t$. Time-consistency of type 1, instead, fixes the whole portfolio whatever is the evaluation time and represents a recursivity property guaranteeing that the capital to be allocated at $s$ to $X$ as a sub-portfolio of $Y$ is the same as the capital to be allocated when we proceed backward in time in two steps and when the whole position $Y$ is fixed.
\normalcolor
\bigskip


We now investigate
the relation between dynamic CARs and dynamic convex risk measures and, in particular,
\normalcolor
the existence of dynamic CARs satisfying some suitable axioms among those listed.\medskip

In the following, let $(\rho_t)_{t\in [0,T]}$ be a dynamic convex risk measure satisfying time-consistency and assume that, for any $t\in [0,T]$,
\begin{assumption}\label{ass:max-realized}
\begin{equation} \label{eq: rhot essmax}
\rho_t(X)= \essmax_{Q \in \mathcal{Q}_t} \{E_Q [-X \vert \mathcal{F}_t] - c_t(Q) \} \quad \mbox{for any } X \in L^{\infty}(\mathcal{F}_T),
\end{equation}
for the minimal penalty term $c_t$ defined in~\eqref{eq: c penalty} and
$\mathcal{Q}_t$ defined accordingly.
\end{assumption}
We recall that Assumption~\ref{ass:max-realized} is fulfilled, for instance, when $\rho_t$ is continuous from below (see Bion-Nadal~\cite{bion-nadal-FS}).

We recall from the static case that the gradient and the subdifferential CARs are related to the optimal scenarios in the dual representation of $\rho=\rho_0$ (see~\cite{delbaen-pisa,centrone-rg,kalkbrener}). Similarly to the static case (see~\cite{centrone-rg}), a dynamic counterpart of the subdifferential CAR can be defined as follows:
%
\begin{equation} \label{eq: lambdat subdiff}
\Lambda^{sub}_t(X;Y) \triangleq  E_{Q_t^Y} [-X \vert \mathcal{F}_t] - c_t(Q_t^Y),
\end{equation}
where
\begin{equation} \label{eq: Qt optimal}
Q_t^Y \in \arg \essmax_{Q \in \mathcal{Q}_t} \{E_Q [-Y \vert \mathcal{F}_t] - c_t(Q) \}.
\end{equation}

Note that, similarly to the static case, the CAR above is not uniquely assigned but represents a family of CARs depending on the choice of $Q_t^Y$ with respect to it is defined. A discussion on this point will follow later.
\medskip

Similarly as in~\cite{delbaen-pisa}, where the static version is considered, we define the subdifferential of $\rho_t$ at $X \in L^{\infty}(\cF_T)$ as
\begin{align*}
   \partial \rho_t(X) \triangleq  \left\{ Q\in \mathcal{Q}_{t}: \ \
    \rho_t(Y) \geq \rho_t(X) + E_Q[-(Y-X)|\mathcal{F}_t]
    \ \text{ for all } Y \in L^\infty(\mathcal{F}_T)  \right\}.
\end{align*}
It is said (see, e.g.,~\cite{delbaen-pisa,zalinescu}) that $\rho_t$ is subdifferentiable at $X \in L^{\infty}(\cF_T)$ if $\partial \rho_t(X)$ is non-empty.

The following lemma extends to the dynamic case the relationship between subdifferential and optimal scenarios of risk measures, well known in the static case (see, among others, Delbaen~\cite{delbaen-pisa}, Ruszczy\'{n}ski and Shapiro~\cite{rusz-shapiro}, and Z\u{a}linescu~\cite{zalinescu}).

\begin{lemma} \label{rem: subdiff}
For a dynamic convex risk measure $\rho_t$ satisfying Assumption~\ref{ass:max-realized}, $\partial \rho_t(X)$ coincides with the set formed by
all the optimal scenarios in \eqref{eq: rhot essmax} (hence $\partial \rho_t(X) \neq \emptyset$ for any $X \in L^\infty(\mathcal{F}_T)$). Moreover, the optimal scenario $Q_t^X$ is uniquely determined whenever $\rho_t(X)$ is Gateaux differentiable at $X$.
\end{lemma}

\begin{proof}
Let us observe that~\eqref{eq: rhot essmax} guarantees that there exists at least one element in $\partial \rho_t(X)$.
In fact,
in this case, for every $Q_t^X \in {\rm arg \,} \essmax_{Q\in \mathcal{Q}_t} \left\{ E_{{Q}_t}[-X| \mathcal{F}_t] -c_t(Q)\right\}$
\begin{align*}
   \rho_t(Y) - \rho_t(X) &=
   \essmax_{Q \in \mathcal{Q}_t}
   \left\{ E_{{Q}_t}[-Y | \mathcal{F}_t] -c_t(Q)\right\}  -
  \left[ E_{Q_t^X}[-X | \mathcal{F}_t] -c_t(Q_t^X)\right]\\
   & \geq
    E_{Q_t^X}[-(Y-X) | \mathcal{F}_t]
\end{align*}
holds for any $Y \in L^{\infty}(\cF_T)$. Thus $Q_t^X$
belongs to $\partial \rho_t(X)$.

Vice versa,
$Q_t^X \in \partial \rho_t(X)$ implies that it is an optimal scenario in the dual representation of $\rho_t(X)$. Indeed, $Q_t^X \in \partial \rho_t(X)$ implies that $E_{Q_t^X} [-Y \vert \mathcal{F}_t] -\rho_t(Y) \leq E_{Q_t^X} [-X \vert \mathcal{F}_t] -\rho_t(X)$ for any $Y \in L^\infty(\mathcal{F}_T)$, hence
\begin{equation*}
c_t(Q_t^X)= {\essmax}_{Y \in L^\infty(\mathcal{F}_T)} \{E_{Q_t^X} [-Y \vert \mathcal{F}_t] -\rho_t(Y) \}=E_{Q_t^X} [-X \vert \mathcal{F}_t] -\rho_t(X).
\end{equation*}
It then follows that
\begin{equation*}
\rho_t(X)=E_{Q_t^X} [-X \vert \mathcal{F}_t] -c_t(Q_t^X),
\end{equation*}
therefore $Q_t^X \in \partial \rho_t(X)$ is an optimal scenario in the dual representation of $\rho_t(X)$.
\smallskip

As in the static case, we now show that $\partial \rho_t(X)$ is a singleton when $\rho_t$ is Gateaux differentiable. To this aim, we say (see, e.g.,~\cite{rusz-shapiro,zalinescu} for the static case) that $\rho_t$ is Gateaux differentiable at $X$ if
$\rho_t$ is directionally differentiable at $X$ in the direction $Y$, i.e., if there exists
the limit in
\begin{equation}\label{eq:limit}
  D\rho_t(X;Y) \triangleq  \lim_{h \to 0} \frac{\rho_t(X+hY)-\rho_t(X)}{h},
\end{equation}
and, for any $X\in {\rm dom}(\rho_t)$, the directional derivative. 
\begin{equation*}
D\rho_t(X;Y)= \lim_{h \searrow 0^+} \frac{\rho_t(X+hY)-\rho_t(X)}{h}
\geq
\lim_{h \searrow 0^+} \frac{E_{Q_t^X} [-hY|\mathcal{F}_t] }{h} =  E_{Q_t^X} [-Y|\mathcal{F}_t],
\end{equation*}
while
\begin{equation*}
D\rho_t(X;Y)= \lim_{h \nearrow 0^-} \frac{\rho_t(X+hY)-\rho_t(X)}{h}
\leq
\lim_{h \nearrow 0^-} \frac{E_{Q_t^X} [-hY|\mathcal{F}_t] }{h} = E_{Q_t^X} [-Y|\mathcal{F}_t].
\end{equation*}
From the uniqueness of the limit in~\eqref{eq:limit} we can conclude that
$$
  D\rho_t(X;Y)= E_{Q_t^X} [-Y|\mathcal{F}_t].
$$
Consequently, for any pair $Q^1_X,Q^2_X \in {\rm arg\,} \essmax_{\mathcal{Q}_t} \left\{ E_{Q}[-X| \mathcal{F}_t] -c_t(Q)\right\}$
the identity
$$
  E_{Q^1_X} [-Y|\mathcal{F}_t] =E_{Q_X^2} [-Y|\mathcal{F}_t]
$$
holds for every $Y \in L^\infty(\mathcal{F}_T)$.
Thus
$Q^1_X$ and $Q^2_X$ coincide
 on $\mathcal{F}_T$ and $D\rho_t(X;Y)$ can be represented by (one of) them.
Hence, $ \partial \rho_t(X)$ is a singleton if $\rho_t$ is Gateaux differentiable at $X$.
\end{proof}
\medskip

Note that, from the arguments above, the dynamic subdifferential CAR $\Lambda^{sub}_t$ can be also written as
\begin{equation} \label{eq: lambdat subdiff-3}
\Lambda^{sub}_t(X;Y)= \rho_t(Y) -E_{Q_t^Y} [-(Y-X) \vert \mathcal{F}_t].
\end{equation}

The previous lemma implies that, for a dynamic convex risk measure satisfying Assumption~\ref{ass:max-realized}, the subdifferential CAR represents a family of CARs since, in general, there may be several optimal scenario $Q_t^X$. The optimal scenario $Q_t^X$ is instead uniquely determined whenever $\rho_t(X)$ is Gateaux differentiable at $X$.

The non-differentiability of a risk measure (leading to a $ \partial \rho_t(X)$ that is not a singleton in general, hence to a family of -static or dynamic- CARs) could seem to be problematic since different CARs may rank differently
sub-portfolios. However, we wish to emphasize that this may happen in general since to any (static or dynamic) risk measure one can associate different capital allocation rules according to different axioms (e.g., gradient or subdifferential, marginal, quantile-based, Aumann-Shapley, ...). See~\cite{dhaene et al,tasche,tsanakas}. In order to define uniquely a CAR (e.g., out of a family of CARs), one possibility could be to choose an optimal capital allocation rule where optimality is taken with respect to a suitable deviation measure (as discussed by Dhaene et al.~\cite{dhaene et al}). Another possibility could be to fix the family of CARs according to some desirable axioms and, in case of non-differentiability of the risk measure and of subdifferential CARs, to choose the optimal scenario on which the CAR is built according to some criteria, e.g. of minimal martingale or entropy (see F\"{o}llmer and Schweizer~\cite{follmer-schweizer}, Frittelli~\cite{frittelli-entropy}, Centrone and Rosazza Gianin~\cite{centrone-rg} for a discussion on related CARs).\medskip
\normalcolor

The following result establishes a one-to-one correspondence between $(\Lambda_t)_{t\in [0,T]}$ and $(\rho_t)_{t\in [0,T]}$ and the relationship between the properties of $(\Lambda_t)_{t\in [0,T]}$ and of $(\rho_t)_{t\in [0,T]}$, in line with Theorem 4.3 of Kalkbrener~\cite{kalkbrener} and of Proposition 4 of Centrone and Rosazza Gianin~\cite{centrone-rg}.

\begin{proposition}\label{prop: from lambda to rho and viceversa}
\noindent a) If $\Lambda_t: L^{\infty}(\mathcal{F}_T) \times L^{\infty}(\mathcal{F}_T) \to L^{\infty}(\mathcal{F}_t)$, for any $t \in [0,T]$, is a monotone and weakly convex functional satisfying $\Lambda_t(X;Y) \leq \Lambda_t(X;X)$ for any $X,Y \in L^{\infty}(\mathcal{F}_T)$ (no-undercut) and time-consistency of type 1 (resp. of type 2), then the associated dynamic risk measure $\rho_t$, defined as $\rho_t(X) \triangleq \Lambda_t(X;X)$ for $X \in L^{\infty}(\mathcal{F}_T)$, is a monotone, convex and weak time-consistent (resp. time-consistent) risk measure. Moreover, if $\Lambda_t$ satisfies also $\mathcal{F}_t$-cash-additivity, then the associated risk measure $\rho_t$ is also cash-additive (hence, a dynamic convex and time-consistent risk measure).\smallskip

\noindent b) If $\rho_t$ is a dynamic convex risk measure satisfying continuity from below,
 then there exists at least a dynamic CAR $\Lambda_t$ (e.g., $\Lambda^{sub}_t$) satisfying monotonicity, weakly convexity, $\mathcal{F}_t$-1-cash-additivity and no-undercut. \smallskip
\end{proposition}

\begin{proof}
a) Monotonicity of $\rho_t$ can be checked easily due to the corresponding properties of $\Lambda_t$.

If $\Lambda_t$ satisfies time-consistency of type 1, then weak time-consistency of $\rho_t$ follows. Indeed, it holds that for any $0\leq s\leq t\leq T$ and any $X \in L^\infty(\cF_T)$
$$
   \rho_s(X)=\Lambda_s(X;X)=\Lambda_s(-\Lambda_t(X;X);X)
   \leq \rho_s(-\rho_t(X)),
$$
where the second equality comes from time-consistency of type 1, while the inequality follows from the no-undercut property.

Time-consistency of type 2 of $\Lambda_t$, instead, implies time-consistency of $\rho_t$. Indeed, for any $0\leq s\leq t\leq T$ and any $X,Y\in L^\infty(\cF_T)$ we have
$$
   \rho_s(X)=\Lambda_s(X;X)=\Lambda_s(-\Lambda_t(X;X);-\rho_t(X))
   = \rho_s(-\rho_t(X)),
$$
where the second equality comes from time-consistency of type 2.
\normalcolor

Convexity: from $\rho_t(X) \triangleq \Lambda_t(X;X)$ it follows that for any $t \in [0,T]$, $\alpha \in [0,1]$ and $X,Y \in L^{\infty}(\mathcal{F}_T)$
\begin{eqnarray*}
\rho_t(\alpha X+ (1-\alpha) Y)&=& \Lambda_t (\alpha X+ (1-\alpha) Y;\alpha X+ (1-\alpha) Y) \\
&\leq & \alpha \Lambda_t (X;\alpha X+ (1-\alpha) Y)+(1-\alpha) \Lambda_t ( Y;\alpha X+ (1-\alpha) Y) \\
&\leq &  \alpha \rho_t (X)+(1-\alpha) \rho_t ( Y),
\end{eqnarray*}
where the former inequality is due to weakly convexity of $\Lambda_t$, while the latter from no-undercut.

Cash-additivity of $\rho_t$ is straightforward under $\mathcal{F}_t$-cash-additivity of $\Lambda_t$.\smallskip

b) By the hypothesis on $\rho_t$, the dual representation~\eqref{eq: rhot essmax} holds true (see Bion-Nadal~\cite{bion-nadal-FS} and the references therein). Consider now
\begin{equation*}
\Lambda_t(X;Y)= \Lambda^{sub}_t(X;Y)=E_{Q_t^Y} [-X \vert \mathcal{F}_t] - c_t(Q_t^Y),
\end{equation*}
where $Q_t^Y$ denotes an optimal scenario in the dual representation of $\rho_t(Y)$.

By definition of $Q_t^Y$, $\Lambda_t^{sub}(Y;Y)=\rho_t(Y)$ and $\Lambda^{sub}_t$ is a dynamic CAR. Monotonicity, weakly convexity and $\mathcal{F}_t$-1-cash-additivity can be easily checked. No-undercut, instead, follows by
\begin{equation*}
\Lambda_t(X;Y)= E_{Q_t^Y} [-X \vert \mathcal{F}_t] - c_t(Q_t^Y) \leq \essmax_{Q \in \mathcal{Q}_t} \{E_Q [-X \vert \mathcal{F}_t] - c_t(Q) \}=\rho_t(X)
\end{equation*}
for any $t \in [0,T]$ and $X,Y \in L^{\infty}(\mathcal{F}_T)$. Furthermore, sub-allocation is also fulfilled. Indeed, for any $Y_1,...,Y_n, Y \in L^{\infty}(\mathcal{F}_T)$ such that $Y=\sum_{i=1}^n Y_i$ and $t \in [0,T]$ it holds that
\begin{eqnarray*}
\Lambda^{sub}_t \left(\sum_{i=1}^n Y_i;Y \right) &=& E_{Q_t^Y} \left[\left. -\sum_{i=1}^n Y_i \right| \mathcal{F}_t\right] - c_t(Q_t^Y) \\
&\geq & \sum_{i=1}^n E_{Q_t^Y} [- Y_i \vert \mathcal{F}_t] - n c_t(Q_t^Y) \\
&=& \sum_{i=1}^n \{ E_{Q_t^Y} [- Y_i \vert \mathcal{F}_t] - c_t(Q_t^Y) \}\\
&=& \sum_{i=1}^n \Lambda^{sub}_t(Y_i;Y).
\end{eqnarray*}
\end{proof}
\medskip

Note that, starting from a dynamic and time-consistent CAR, we can obtain a time-consistent dynamic risk measure. The converse implication, instead, needs a further investigation. A further step will be then to investigate the existence of a time-consistent CAR induced by a time-consistent dynamic risk measure.
For a $\rho_t$ coming from a $g$-expectation we will show in the next section that the existence is guaranteed.
\medskip

\section{CARs associated to a $\rho_t$ induced by a $g$-expectation}\label{sec:g-expectation}

In the following, we will restrict our attention to dynamic risk measures that are induced by $g$-expectations. The main motivations for this choice can be summarized as follows. First, as shown in~\cite{coquet-et-al,peng-05,RG}, a wide family of dynamic time-consistent convex risk measures satisfying some further assumptions come from a $g$-expectation.
Second and somehow related to the previous point, there is a huge literature on risk measures induced by BSDEs in a Brownian setting or in a setting with jumps. See, among others, Barrieu and El Karoui~\cite{barrieu-elkaroui}, Delbaen et al.~\cite{DPR}, Rosazza Gianin~\cite{RG}, Laeven and Stadje~\cite{laeven-stadje}, Quenez and Sulem~\cite{quenez-sulem}, Calvia and Rosazza Gianin~\cite{calvia-rg}. Finally, some recent works (see Kromer and Overbeck~\cite{kromer-overbeck-bsde,kromer-overbeck-volterra} and Mabitsela et al.~\cite{mabitsela-etal}) already focus on dynamic risk measures induced by BSDEs and Volterra equations to investigate the gradient allocation.
\bigskip

Let $(\rho_t)_{t \in [0,T]}$ be a dynamic convex and time-consistent risk measure that is induced by a $g$-expectation in a Brownian setting, i.e., $(\rho_t(X), Z_t^X)$ solves the following BSDE
\begin{equation} \label{eq: rhot - g-exp}
\rho_t(X)=-X+ \int_t ^T g_{\rho} (s,Z_s^X) \, ds - \int_t^T Z_s^X \, dB_s, \quad X \in L^{\infty}(\mathcal{F}_T),
\end{equation}
for a suitable driver $g_{\rho}$ satisfying the usual assumptions and convexity in $z$. In particular, $g_{\rho}$ is uniformly Lipschitz in $z$ and $g_{\rho}(s,0)=0$ for any $s \in [0,T]$.


Inspired by Kromer and Overbeck~\cite{kromer-overbeck-bsde} where the authors proved that, under Gateaux differentiability of $\rho_t$, the gradient CAR for a $\rho_t$ as in~\eqref{eq: rhot - g-exp} satisfies a BSDE with a driver depending on the gradient of $g_{\rho}$, we now introduce a general formulation of dynamic CARs also induced by a $g$-expectation but going beyond the gradient approach and under weaker assumptions than differentiability.
Assume now that also the dynamic CAR $\Lambda_t$ we are looking for is induced by a $g$-expectation with a different driver $g_{\Lambda}$. More precisely, assume that
\begin{equation} \label{eq: lambdat - g-exp}
\Lambda_t(X;Y)=-X+ \int_t ^T g_{\Lambda} (s,Z_s^{X,Y}, Z_s^Y) \, ds - \int_t^T Z_s^{X,Y} \, dB_s,
\end{equation}
where $Z_s^{X,Y}$ is part of the solution while $Z_s^Y$ comes from
\begin{equation*}
\rho_t(Y)=-Y+ \int_t ^T g_{\rho} (s,Z_s^Y) \, ds - \int_t^T Z_s^Y \, dB_s,
\end{equation*}
and with $g_{\Lambda}(s,z,z^y)$ satisfying (for any $z^y \in \Bbb R^d$) the usual non-normalized assumptions on $z$ and the condition
\begin{equation} \label{eq: glambda -grho}
g_{\Lambda}(s,z,z)= g_{\rho} (s,z) \quad \mbox{for any } s\in [0,T], z \in \mathbb{R}^d.
\end{equation}
In this case, the driver $g_{\Lambda}(s,z,z^y)$ depends then on an additional parameter $z^y$. Note that the assumptions on $g_{\Lambda}$ imply the existence and uniqueness of the solution $(\Lambda_t(X;Y);Z_t^{X,Y})_{t \in [0,T]}$ of~\eqref{eq: lambdat - g-exp}, while condition~\eqref{eq: glambda -grho} guarantees that $\Lambda_t$ is a dynamic CAR, i.e., $\Lambda_t(Y;Y)=\rho_t(Y)$ for any $t \in [0,T]$ and $Y \in L^{\infty} (\mathcal{F}_T)$.

Given a dynamic risk measure $(\rho_t)_{t \in [0,T]}$ induced by a $g_{\rho}$-expectation, it is then possible to define several dynamic capital allocations induced by a $g_{\Lambda}$-expectation with $g_{\Lambda}$ fulfilling condition~\eqref{eq: glambda -grho}. Viceversa, given a dynamic family $(\Lambda_t)_{t \in [0,T]}$ induced by a $g_{\Lambda}$-expectation, the associated dynamic risk measure is uniquely determined via \eqref{eq: glambda -grho}.

As already discussed above, the assumption of a dynamic CAR induced by a $g$-expectation generalizes the gradient case and, in view of a result of Coquet et al.~\cite{coquet-et-al} (see also Remark~\ref{rem: rem BSDE} below), seems to be rather reasonable for risk measures coming from $g$-expectations.

\begin{remark} \label{rem: rem BSDE}
a) If $(\rho_t)_{t\in [0,T]}$ is Gateaux differentiable at any time $t\in [0,T]$, then the gradient allocation $(\Lambda_t^{grad})_{t\in [0,T]}$ is of the form \eqref{eq: lambdat - g-exp} with
\begin{equation} \label{eq: g-gradient allocation}
g_{\Lambda} (s,Z_s^{X,Y}, Z_s^Y)= \nabla g_\rho(s,Z_s^Y) \cdot Z_s^{X,Y},
\end{equation}
where, in the $d$-dimensional case, $\nabla g_\rho$ stands for the gradient of $g_\rho$ with respect to $z$.
We observe that this result has been already proved in~\cite[Theorem 3.1]{kromer-overbeck-bsde} for BSDEs where $\nabla  g_\rho (s,Z_s^Y) $ is from a BMO\footnote{Note that in~\cite{kromer-overbeck-bsde} $g_{\rho}$ may have quadratic growth in $z$, hence $\nabla  g_\rho$ is not bounded in general. $\nabla  g_\rho (s,Z_s^Y)$ is then assumed to be from a BMO in order to be able to apply Girsanov Theorem and to define an equivalent probability measure in terms of $\nabla  g_\rho$. For a precise definition of BMO we refer to Kazamaki~\cite{kazamaki}.}.

b) It seems quite reasonable to assume that a dynamic CAR associated to a dynamic risk measure as in \eqref{eq: rhot - g-exp} is induced by a $g_{\Lambda}$-expectation. Indeed, let $(\rho_t)_{t\in [0,T]}$ be as in~\eqref{eq: rhot - g-exp} and let $(\Lambda_t)_{t\in [0,T]}$ be convex, monotone and cash-additive in $X$, and satisfying time-consistency. Assume, in addition, that $\Lambda_t$ satisfies $\mathcal{E}_{\bar{\mu}}$-dominance for some $\bar{\mu}>0$, that is, \begin{equation*}
\Lambda_t(X;Y) -\Lambda_t(Z;Y) \leq  \mathcal{E}_{\bar{\mu}}(-(X-Z) \vert \mathcal{F}_t), \quad \mbox{ for any } X,Y,Z \in L^{\infty}(\mathcal{F}_T),t\in [0,T].
\end{equation*}
It then follows by Coquet et al.~\cite{coquet-et-al} that $(\Lambda_t)_{t\in [0,T]}$ is also induced by a $\bar{g}$-expectation for some $\bar{\mu}$-Lipschitz driver $\bar{g}=\bar{g}_Y$, i.e., there exists some $\bar{g}_Y(s,z) \leq \bar{\mu} \vert z \vert$ such that $\Lambda_t(X;Y) =  \mathcal{E}_{\bar{g}_Y}(-X \vert \mathcal{F}_t)$ for any $X,Y \in L^{\infty} (\mathcal{F}_T)$.

Note that $\mathcal{E}_{\mu}$-dominance of $(\Lambda_t)_{t\in [0,T]}$ is guaranteed, for instance, when
\begin{equation*}
\Lambda_t(X;Y) -\Lambda_t(Z;Y) \leq  \rho_t(X-Z), \quad \mbox{ for any } X,Y,Z \in L^{\infty}(\mathcal{F}_T), t\in [0,T].
\end{equation*}
The condition above seems to be a reasonable generalization of
\begin{equation*}
\Lambda_t(X;Y) \leq \rho_t(X)=\mathcal{E}_{g_{\rho}}(-X \vert \mathcal{F}_t) \leq \mathcal{E}_{\mu}(-X \vert \mathcal{F}_t), \quad \text{for any } X,Y \in L^{\infty}(\mathcal{F}_T), t\in [0,T],
\end{equation*}
that is automatically fulfilled under no-undercut.
\normalcolor
\end{remark}

Assume now that $\rho_t$ is convex but non necessarily Gateaux differentiable
and assume that $\partial \rho_t(X) \neq \emptyset$, that is, $\rho_t$ is subdifferentiable at $X$ (see Lemma~\ref{rem: subdiff} and~\cite{rusz-shapiro} for a detailed discussion).
As from Lemma~\ref{rem: subdiff}, we can identify an element of $\partial \rho_t(X)$ with the corresponding probability measure $Q$.

%
%

We are now ready to prove the existence of a dynamic CAR that is time-consistent for a dynamic time-consistent risk measure induced by a $g$-expectation, as a follow-up of Proposition~\ref{prop: from lambda to rho and viceversa}-b).

With an abuse of notation, formulations containing the term $\partial g_{\rho}$ should be read as follows: to any element of the set $\partial g_{\rho}$ we associate an element (e.g., of $\partial \rho_t(X)$, $\Lambda_t^{sub}$ or $g_{\Lambda}$). To be more precise, in \eqref{eq: Qt optimal-subdiff g} to any element of the family $(\partial g_{\rho}(u,Z_u^X))_{u \in [0,t]}$ of the subdifferential of $g_{\rho}$ it corresponds an element $Q_t^X$ of the subdifferential $\partial \rho_t(X)$ and, consequently, a different element of the family $\Lambda^{sub}$. Or, better, similarly to the general case, $\Lambda^{sub}$ is not uniquely defined but corresponds to a family of CARs depending on the choice of the optimal scenario to which is associated. A discussion on the subdifferentiability of $g_{\rho}$ is postponed to Remark~\ref{rem: subdif g}. \smallskip
\normalcolor

\begin{proposition}[Existence of a CAR satisfying time-consistency] \label{prop: existence-subdiff}
Let $(\rho_t)_{t\in[0,T]}$ be a dynamic convex risk measure that is induced by a $g_{\rho}$-expectation as in~\eqref{eq: rhot - g-exp}.

If $g_{\rho}$ is subdifferentiable, then $\rho_t$ is also subdifferentiable and elements $Q_t^X$ defined by
\begin{equation} \label{eq: Qt optimal-subdiff g}
E\left[ \left.\frac{dQ_t^X}{dP} \right| \mathcal{F}_t\right]  \triangleq \mathcal{E} (\partial g_{\rho}(t,Z_t^X) \cdot B)= \exp \left\{-\frac 12 \int_0^t \Vert \partial g_{\rho}(u,Z_u^X)\Vert ^2 \, d u - \int_0^t \partial g_{\rho}(u,Z_u^X) \, dB_u \right\}
\end{equation} belong to $\partial \rho_t(X)$.

Furthermore, $\Lambda^{sub}_t(X;Y)$ related to $Q_t^X$ as above is a (family of) time-consistent dynamic CAR(s) satisfying monotonicity, no-undercut and sub-allocation. Existence of a CAR satisfying the axioms above is therefore guaranteed.

Moreover, $\Lambda^{sub}_t(X;Y)$ satisfies the following BSDE
\begin{equation} \label{eq: lambdat-sub- bsde}
\Lambda^{sub}_t(X;Y)=-X + \int_t ^T \left[g_{\rho}(s,Z_s^Y)+\partial g_{\rho}(s,Z_s^Y)\cdot  (Z_s^{X,Y} -Z_s^Y ) \right] \, ds - \int_t^T Z_s^{X,Y} \, dB_s.
\end{equation}
\end{proposition}

\begin{proof}
Let $g_{\rho}$ be subdifferentiable and convex and let $\bar{q} \in \partial g_{\rho}(s,z)$, where $\partial g_{\rho}(s,z) \triangleq \{q \in \bR ^d: g_{\rho}(s,u) \geq g_{\rho}(s,z) + q \cdot (u-z) \mbox{ for any } u \in \bR ^d\}$. It then holds that $g_\rho$ is continuous in the second variable and $g_{\rho}(s,z)=\bar{q} \cdot z- g^*_{\rho}(s,\bar{q})$ where $g^*_{\rho}$ is the convex conjugate of $g_{\rho}$. Note that any element $\bar{q} \in \partial g_{\rho}(s,z)$ is such that $\Vert \bar{q} \Vert \leq \mu$ where $\mu >0$ is the Lipschitz constant of $g_{\rho}$. This follows because $g^*_{\rho}(s,\bar{q})=+\infty$ for $\Vert \bar{q} \Vert >\mu$ (see, e.g.,~\cite{DPR}, Prop. 3.6) and $g_{\rho}(s,z)=\bar{q} \cdot z- g^*_{\rho}(s,\bar{q})$ is finite. In the following, we sometimes denote $g=g_{\rho}$ when there is no possible misunderstanding. With a slight abuse of notation, we also indicate
an element in the subdifferential of $g(s,z)$ by $\partial g(s,z)$.

We start proving that $\rho_t$ is also subdifferentiable and that $Q_t^X$ in~\eqref{eq: Qt optimal-subdiff g} belongs to $\partial \rho_t(X)$.
Fix now $X \in L^{\infty}(\mathcal{F}_T)$ and consider any $Y \in L^{\infty}(\mathcal{F}_T)$. By~\eqref{eq: rhot - g-exp},
\begin{eqnarray*}
\rho_t(Y) - \rho_t(X) &=&  -(Y-X) + \int_t ^T \left[ g_{\rho} (s,Z_s^Y)-g_{\rho} (s,Z_s^X) \right]\, ds - \int_t^T \left( Z_s^Y- Z_s^X \right) \, dB_s \\
   &\geq & -(Y-X) + \int_t ^T \partial g(s,Z_s^X ) \cdot \left(Z_s^Y - Z_s^X \right)\, ds - \int_t^T \left( Z_s^Y- Z_s^X \right) \, dB_s \\
   &\geq & -(Y-X) - \int_t ^T  \left(Z_s^Y - Z_s^X \right)\, d\bar{B}^{Q^X}_s
\end{eqnarray*}
where $d\bar{B}^{Q^X}_s= dB_s - \partial g(s,Z_s^X ) \, ds$ and, by Girsanov Theorem, $(\bar{B}^{Q^X}_t)_{t \in [0,T]}$ is a $Q^X$-Brownian motion. By taking the conditional expectation with respect to $Q_t^X$ of the first and last term in the chain of inequalities above, it follows that
\begin{equation*}
\rho_t(Y) - \rho_t(X) \geq E_{Q_t^X} [-(Y-X) \vert \mathcal{F}_t] \quad \mbox{for any } Y\in L^{\infty}(\mathcal{F}_T),
\end{equation*}
hence $\rho_t$ is subdifferentiable at $X$ and $Q_t^X \in \partial \rho_t(X)$ is an optimal scenario.

It is easy to check that if $\rho_t$ is subdifferentiable at $X \in L^{\infty}(\cF_T)$, then $\rho_t$ is continuous from above at $X$.
By the arguments above and by the characterization of the penalty term in Delbaen et al.~\cite{DPR}, Theorem 3.2, and Barrieu and El Karoui~\cite{barrieu-elkaroui}, Theorem 7.4, it then follows that
\begin{eqnarray*}
\rho_t(X) &=&  E_{Q_t^X} [-X \vert \mathcal{F}_t]-c_t(Q_t^X) \\
&=&  E_{Q_t^X} [-X \vert \mathcal{F}_t]- E_{Q_t^X} \left[\left. \int_t^T g^*_{\rho} (s,\partial g(s,Z_s^X )) \, ds \right| \mathcal{F}_t\right]\\
&=&  E_{Q_t^X} \left[\left. -X - \int_t^T \left[ \partial g(s,Z_s^X ) \cdot Z_s^X -g(s,Z_s^X) \right] \, ds \right| \mathcal{F}_t \right].
\end{eqnarray*}
By similar arguments and by the martingale representation theorem, there exists a unique stochastic process $(Z_s^{X,Y})_{s \in [0,T]}$ such that
\begin{eqnarray*}
\Lambda^{sub}_t(X;Y) &=&  E_{Q_t^Y} [-X \vert \mathcal{F}_t]-c_t(Q_t^Y) \\
&=&  E_{Q_t^Y} \left[-X - \int_t^T \left[ \partial g(s,Z_s^Y ) \cdot Z_s^Y -g(s,Z_s^Y) \right] \, ds \vert \mathcal{F}_t \right] \\
&=&  -X - \int_t^T \left[ \partial g(s,Z_s^Y ) \cdot Z_s^Y -g(s,Z_s^Y) \right] \, ds - \int_t ^T Z_s^{X,Y} \, d\bar{B}^{Q^Y}_s  \\
&=&  -X + \int_t^T \left[ \partial g(s,Z_s^Y ) \cdot (Z_s^{X,Y} - Z_s^Y) +g(s,Z_s^Y)  \right] \, ds - \int_t ^T Z_s^{X,Y} \, dB_s ,
\end{eqnarray*}
hence $(\Lambda^{sub}_t(X;Y), Z_t^{X,Y})$ solves a $g_{\Lambda}$-expectation with final condition $-X$ and driver
\begin{equation*}
g_{\Lambda}(s,Z_s^{X,Y},Z_s^Y)= \partial g(s,Z_s^Y ) \cdot (Z_s^{X,Y} - Z_s^Y) +g(s,Z_s^Y)
\end{equation*}
depending on $Z_s^Y$. It is straightforward to check that $g_{\Lambda} (s,0,Z_s^Y)=g(s,Z_s^Y)-\partial g(s,Z_s^Y ) \cdot Z_s^Y$ satisfies $(g_2)$ since $g$ satisfies condition~\eqref{g-condition} and, for any $s\in [0,T]$, $z\in \bR^d$, $q \in\partial g(s,z)$ is such that $\Vert q\Vert \leq \mu$
where $\mu>0$ is the Lipschitz constant of $g$.
We now prove time-consistency of type 1 and 2, i.e.
$$\Lambda_{s}^{sub}(-\Lambda_t^{sub}(X;Y);Y)=\Lambda_s^{sub}(X;Y) \quad \text{and} \quad
\Lambda_{s}^{sub}(-\Lambda_t^{sub}(X;Y);-\rho_t(Y))=\Lambda_s^{sub}(X;Y),$$
for any $0\leq s\leq t\leq T$ and any $X,Y\in L^\infty$.
We observe that
\begin{align*}
   \Lambda_s^{sub}(X;Y)&=
   -X + \int_{s}^T g_\Lambda(s,Z_s^{X,Y},Z_s^Y) \, ds
   -\int_s^T Z^{X,Y}_s dB_s\\
   & = -X + \int_{t}^T g_\Lambda(s,Z_s^{X,Y},Z_s^Y) \, ds
   -\int_t^T Z^{X,Y}_s dB_s \\
   & \ \ \ \ + \int_{s}^t g_\Lambda(s,Z_s^{X,Y},Z_s^Y) \, ds
   -\int_s^t Z^{X,Y}_s dB_s\\
   &= \Lambda_t^{sub}(X;Y)  + \int_{s}^t g_\Lambda(s,Z_s^{X,Y},Z_s^Y) \, ds
   -\int_s^t Z^{X,Y}_s dB_s.
\end{align*}
Moreover, $ \Lambda_s^{sub} (-\Lambda_t^{sub}(X;Y);-\rho_t(Y) )$ solves the following BSDE
\begin{align*}
  \Lambda_s^{sub} (-\Lambda_t^{sub}(X;Y);-\rho_t(Y) )&=
  \Lambda_t^{sub}(X;Y) + \int_{s}^t g_\Lambda(s,\bar{Z}_s,Z_s^{-\rho_t(Y)}) \, ds
   -\int_s^t \bar{Z}_s dB_s\\
    &=\Lambda_t^{sub}(X;Y) + \int_{s}^t g_\Lambda(s,\bar{Z}_s,Z_s^{Y}) \, ds
   -\int_s^t \bar{Z}_s dB_s.
\end{align*}
where the last equality follows from $Z_s^{-\rho_t(Y)}= Z_s^Y$, which is due to time-consistency of $(\rho_t)_{t\in [0,T]}$.


\textcolor{black}{We then conclude that $ \Lambda_s^{sub} (-\Lambda_t^{sub}(X;Y);-\rho_t(Y) )=\Lambda_s^{sub} (-\Lambda_t^{sub}(X;Y);Y )=  \Lambda_s^{sub}(X;Y)$, i.e. $\Lambda_s^{sub} $ satisfies both time-consistency of type 1 and 2.
}
Monotonicity, no-undercut and sub-allocation have been already proved in Proposition~\ref{prop: from lambda to rho and viceversa}-b).
\end{proof}
\medskip

\begin{remark}
a) For dynamic
\color{black}coherent differentiable risk measures induced by $g_{\rho}$-expectations with differentiable $g_{\rho}$, the result above reduces to Theorem 3.1 of Kromer and Overbeck~\cite{kromer-overbeck-bsde}.
 Indeed, if $(\rho_t)_{t\in [0,T]}$ is coherent, then the term corresponding to the penalty term disappears. It is then easy to check that differentiability of $g_{\rho}$ implies that also $\rho_t$ is Gateaux differentiable
 \normalcolor at any time $t\in [0,T]$ and $\Lambda^{sub}_t(X;Y)=\Lambda^{grad}_t(X;Y)$ satisfies a $g_{\Lambda}$-expectation with final condition $-X$ and driver
\begin{equation*}
g_{\Lambda}(s,Z_s^{X,Y},Z_s^Y)= \partial g_{\rho}(s,Z_s^Y ) \cdot Z_s^{X,Y} =\nabla g_{\rho}(s,Z_s^Y ) \cdot Z_s^{X,Y}.
\end{equation*}
\color{black}As recalled previously, in~\cite{kromer-overbeck-bsde} $\nabla  g_\rho (s,Z_s^Y)$ is assumed to be from a BMO because $g_{\rho}$ may have quadratic growth in $z$. In our case, instead, the BMO assumption can be dropped since $g_{\rho}$ is Lipschitz in $z$, hence $\partial g$ is bounded.
\normalcolor

b) \color{black}
Note that $Z_s^{X,Y}$ in~\eqref{eq: lambdat-sub- bsde} coincides with $Z_s^X$ when $X=Y$. Indeed, for any $X \in L^{\infty} (\mathcal{F}_T)$
\begin{eqnarray*}
\Lambda^{sub}_t(X;X) &=& \rho_t(X) \\
&=& -X+ \int_t ^T g_{\rho} (s,Z_s^X) \, ds - \int_t^T Z_s^X \, dB_s \\
&=& -X+ \int_t ^T \left[g_{\rho}(s,Z_s^X)- \partial g_{\rho}(s,Z_s^X ) \cdot (Z_s^X - Z_s^X) \right] \, ds - \int_t^T Z_s^X \, dB_s.
\end{eqnarray*}
By~\eqref{eq: lambdat-sub- bsde} and from the uniqueness of the solution of a BSDE, it then follows that $Z_s^{X,X}=Z_s^X$.
\normalcolor

c) An alternative proof of the previous result can be driven by using the formulation
\begin{equation*}
\Lambda^{sub}_t(X;Y)= \rho_t(Y) -E_{Q_t^Y} [-(Y-X) \vert \mathcal{F}_t]
\end{equation*}
and by applying the martingale representation theorem to both the terms $E_{Q_t^Y} [-Y \vert \mathcal{F}_t]$ and $E_{Q_t^Y} [X \vert \mathcal{F}_t]$.
\end{remark}

\color{black} In the following, we provide some particular cases of dynamic CARs from a $g_{\Lambda}$-expectation and of the corresponding drivers $g_{\Lambda}$.
\normalcolor

\begin{remark}
Assume that $\rho_t$ is a dynamic convex risk measure that is induced by a $g_{\rho}$-expectation.
The gradient, subdifferential and marginal CAR are compatible with the formulation of dynamic CARs by means of $g_{\Lambda}$-expectations and, as discussed below, can be obtained by choosing a suitable driver $g_{\Lambda}$. See~\cite{delbaen-pisa,denault,centrone-rg,kalkbrener,tasche,tsanakas} for the static versions. \medskip

\emph{Gradient case:} if $\rho_t$ is Gateaux differentiable, then Kromer and Overbeck~\cite{kromer-overbeck-bsde} showed that $\Lambda^{grad}_t(X;Y)$ solves a BSDE with driver
\begin{equation*}
g^{grad}_{\Lambda}(s,z,z^y)= \nabla g_{\rho} (s,z^y) \cdot z.
\end{equation*}

\emph{Subdifferential case:} if $\rho_t$ is only subdifferentiable, then the previous result shows that $\Lambda^{sub}_t(X;Y)$ solves a BSDE with driver
\begin{equation*}
g^{sub}_{\Lambda}(s,z,z^y)= \partial g_{\rho}(s,z^y )\cdot (z - z^y) +g(s,z^y),
\end{equation*}
satisfying condition~\eqref{eq: glambda -grho}. As already pointed out previously, $\partial g_{\rho}(s,z^y )$ is not necessarily unique but it is assumed to be chosen and fixed for any $s$ and $z^y$.\smallskip

\emph{Marginal case:} the marginal dynamic $\Lambda^{marg}_t(X;Y)= \rho_t(Y)-\rho_t (Y-X)$ solves
\begin{eqnarray*}
\Lambda^{marg}_t(X;Y) &=& \rho_t(Y)-\rho_t (Y-X) \\
&=&  -X + \int_t ^T  \left[ g_{\rho} (s,Z_s^Y) - g_{\rho} (s,Z_s^{Y-X}) \right] \, ds - \int_t^T \left( Z_s^Y -Z_s^{Y-X} \right) \, dB_s \\
&=&  -X + \int_t ^T  g_{\Lambda}^{\color{black} marg} (s,Z_s^{X,Y}, Z_s^Y) \, ds - \int_t^T Z_s^{X,Y} \, dB_s \\
\end{eqnarray*}
by denoting $Z^{X,Y}_s \triangleq Z_s^Y -Z_s^{Y-X}$ and $g^{marg}_{\Lambda} (s,Z_s^{X,Y}, Z_s^Y) = g_{\rho} (s,Z_s^Y) - g_{\rho} (s,Z_s^{Y}-Z^{X,Y}_s)$.  The driver $g_{\Lambda}^{marg}$ satisfies condition~\eqref{eq: glambda -grho} when $g_{\rho}(s,0)=0$ holds for any $s \in [0,T]$.
\end{remark}

%
%
%
%

In the following, we investigate under which conditions a dynamic CAR $(\Lambda_t)_{t\in [0,T]}$ that is induced by a $g_{\Lambda}$-expectation satisfies some further axioms and vice versa.

\begin{proposition} \label{prop: suff-nec condition glambda - axioms}
Let $g_{\Lambda}$ be a driver satisfying the usual non-normalized assumptions and condition~\eqref{eq: glambda -grho}.\smallskip

i) $\Lambda_t(X+c_t;Y)=\Lambda_t(X;Y)-c_t$ for any $X,Y \in L^{\infty}(\mathcal{F}_T)$, $c_t \in L^{\infty}(\mathcal{F}_t)$.

ii) if $g_{\Lambda}(s,0,z^y) =0$ for any $s \in [0,T]$, $z^{y} \in \bR^d$, then $\Lambda_t (0;Y)=0$  for any $t \in [0,T]$ and $Y \in L^{\infty}(\mathcal{F}_T)$.

iii) if $g_{\Lambda}(s,z,z^y) \leq g_{\rho} (s,z)$ for any $s \in [0,T]$ and $z, z^y \in \bR^d$, then $\Lambda_t$ satisfies no-undercut.

iv)
{\color{black}$\Lambda_t$ satisfies monotonicity.}

v) if $\sum_{i=1}^n g_{\Lambda}(s,z^{i},z^y) \leq g_{\Lambda}(s,\sum_{i=1}^n z^{i},z^y)$ for any $s \in [0,T]$, $z^{i}, z^y \in \bR^d$, then $\Lambda_t$ satisfies sub-allocation.

vi) if $g_{\Lambda}(s,z,z^y)$ is convex in $z$ for any $s \in [0,T]$, $z^y \in \bR^d$, then $\Lambda_t$ satisfies weak convexity.\smallskip

\noindent Moreover, the converse implications hold in iii)-iv) if $g_{\Lambda}(s,z,z^y)$ is continuous in $s$.
\end{proposition}

\begin{proof}
i) and ii) are straightforward thanks to the properties of $g$-expectations with a driver independent on the $y$-component (see Peng~\cite{peng-97} and Rosazza Gianin~\cite{RG}). In particular, ii) is due to the fact that $(\Lambda_t(0;Y);Z_t^{0,Y})_t$ with $Z_t^{0,Y} =0$ solves the BSDE providing $\Lambda_t(0;Y)$.

The proof of the sufficient conditions iii)-vi) follows by Comparison Theorem of BSDEs (see Theorem 2.2 of El Karoui et al.~\cite{EPQ}) and by the properties of $g$-expectations with a driver independent on the $y$-component (see Peng~\cite{peng-97} and Rosazza Gianin~\cite{RG}), while the last statement (concerning necessary conditions) follows by Theorem 4.1 of Briand et al.~\cite{BCHMP}.
\end{proof}

Note that conditions in v) and in vi) are only sufficient but not necessary for sub-allocation and weak convexity, respectively.

It is well known that for convex risk measures {\color{black} that are not coherent} the gradient approach fails to satisfy no-undercut. At the level of a convex $g_{\Lambda}$, we have indeed that the sufficient condition in iii) is not satisfied in general because
$$
g^{grad}_{\Lambda}(s,z,z)= \nabla g_{\rho} (s,z) \cdot z= g_{\rho} (s,z) + g^*_{\rho} (s,\nabla g_{\rho} (s,z)) \geq g_{\rho} (s,z)
$$
for any $z \in \bR^d$, where the equality holds iff $g^*_{\rho} (s,\nabla g_{\rho} (s,z))=0$.
\smallskip

As discussed previously, a popular static CAR related to cooperative game theory is the Aumann-Shapley CAR {\color{black} (see Aumann and Shapley~\cite{aumann-shapley}, Tsanakas~\cite{tsanakas}, Centrone and Rosazza Gianin~\cite{centrone-rg})}.
{\color{black} We now introduce} the dynamic (generalized) Aumann-Shapley CAR defined $\omega \times \omega$ as
\begin{equation} \label{eq: dynamic AS}
    \Lambda_t^{AS}(X;Y)=\int_0^1 E_{Q_t^{\gamma Y}}[-X \vert \mathcal{F}_t] \, d\gamma
\end{equation}
and the
    dynamic penalized Aumann-Shapley CAR with penalty function $c$ as
\begin{equation} \label{eq: dynamic AS penalized}
    \Lambda_t^{c-AS}(X;Y)=\int_0^1 \left[ E_{Q_t^{\gamma Y}}[-X \vert \mathcal{F}_t]-c_{t}(Q_t^{\gamma Y}) \right] \, d\gamma=\int_0^1 \Lambda^{sub}_t(X;\gamma Y) \,d\gamma,
\end{equation}
with $Q_t^{\cdot}$ defined as in~\eqref{eq: Qt optimal}.

The following result deals with the dynamic (penalized) Aumann-Shapley and its time-consistency.

\begin{proposition} \label{prop: as-dynamic}
Let $\rho_t$ be a dynamic time-consistent convex risk measure that is subdifferentiable \textcolor{black}{and such that
the map $G_t(\gamma)\triangleq \rho_t(\gamma X)$ is differentiable in $\gamma\in [0,1]$ for any $X \in L^{\infty}(\mathcal{F}_T)$}.

a) $\Lambda_t^{AS}$ is a {\color{black}dynamic CAR}.
 Furthermore,
\begin{equation} \label{eq: lambda AS as expect}
\Lambda_t^{AS}(X;Y)=E_P\left[\left. -\tilde{L}^Y(T;t) X \right|\mathcal{F}_t\right],
\end{equation}
where
\begin{equation} \label{eq: L tilde}
\tilde{L}^Y (T;t) \triangleq \int_0^1 L^{\gamma Y} (T;t) \, d\gamma= \int_0^1 \frac{\mathcal{E} (\partial g(s ,Z_s^{\gamma Y}) \cdot B)(T)}{\mathcal{E} (\partial g(s,Z_s^{\gamma Y}) \cdot B)(t)} \, d\gamma
\end{equation}
and \begin{equation*}
L^{H}(T;t) \triangleq \frac{\frac{dQ_T^H}{dP}}{\frac{dQ_t^H}{dP}}=\exp \left\{-\frac 12 \int_t^T \Vert \partial g(u,Z_u^H)\Vert ^2 \, d u - \int_t^T \partial g(u,Z_u^H) \, dB_u \right\}.
\end{equation*}

b) If $g_{\rho}$ is positively homogeneous (hence $\rho_t$ is coherent), then $\Lambda_t^{AS}$ is time-consistent.

c) The penalized $\Lambda_t^{c-AS}$ is a dynamic audacious CAR (not a dynamic CAR in general) satisfying no-undercut.
\end{proposition}

\begin{proof}

a) We start proving that $\Lambda_t^{AS}$ is a CAR and later that~\eqref{eq: lambda AS as expect} is satisfied. This proof is in line but extends the one of Corollary 4.1 of Kromer and Overbeck~\cite{kromer-overbeck-bsde} (to non-differentiable risk measures and to the expression of $\Lambda_t$) and of Proposition 9(a) of Centrone and Rosazza Gianin~\cite{centrone-rg} (to the dynamic case).

\textcolor{black}{Let $Q_t^{\gamma X} \in \partial \rho_t(\gamma X)$. Then, by definition,
\begin{eqnarray*}
G'_{t,-} (\gamma) &= & \lim_{h\nearrow 0} \frac{\rho_t(\gamma X + h X)- \rho_t (\gamma X)}{h} \leq E_{Q_t^{\gamma X}} [-X | \mathcal{F}_t] \\
G'_{t,+} (\gamma) &= & \lim_{h\searrow 0} \frac{\rho_t(\gamma X + h X)- \rho_t (\gamma X)}{h} \geq E_{Q_t^{\gamma X}} [-X | \mathcal{F}_t].
\end{eqnarray*}
Under the assumption of differentiability of $G_t(\gamma)$
 for every $\gamma \in [0,1]$, it follows that $G'_{t,-} (\gamma) =G'_{t,+} (\gamma)$
and hence}
the following equality holds {\color{black}true}
\begin{equation*}
\int_0^1 G'_{t,-} (\gamma) d \gamma = \int_0^1 E_{Q_t^{\gamma X}} [-X | \mathcal{F}_t] d \gamma = \int_0^1 G'_{t,+} (\gamma) d \gamma.
\end{equation*}
By normalization of $\rho_t$, $\int_0^1 E_{Q_t^{\gamma X}} [-X | \mathcal{F}_t] d \gamma=G_t(1)-G_t(0)=\rho_t(X)$ and, consequently, $\Lambda_t^{AS}$ is a dynamic CAR.

Furthermore, by Proposition~\ref{prop: existence-subdiff}
and~\eqref{eq: Qt optimal-subdiff g}, we know that $E_P\left[\frac{dQ_t^X}{dP}\right]= \mathcal{E} (\partial g(t,Z_t^X) \cdot B)=$\linebreak$=\exp \left\{-\frac 12 \int_0^t \Vert \partial g(u,Z_u^X)\Vert ^2 \, d u - \int_0^t \partial g(u,Z_u^X) \, dB_u \right\}$. By Fubini-Tonelli Theorem, this implies that
\begin{eqnarray*}
\Lambda_t^{AS}(X;Y)&=& \int_0^1 E_{Q_t^{\gamma Y}}[-X \vert \mathcal{F}_t] \, d\gamma \\
&=& \int_0^1 E_{P}[\left. -X L^{\gamma Y} (T;t) \right| \mathcal{F}_t] \, d\gamma \\
&=& E_{P}\left[\left. \int_0^1(-X L^{\gamma Y} (T;t)) \, d\gamma \right| \mathcal{F}_t \right] \\
&=& E_{P}\left[\left. -X \int_0^1 L^{\gamma Y} (T;t) \, d\gamma \right| \mathcal{F}_t \right] \\
&=& E_{P}\left[\left. -X  \tilde{L}^{Y} (T;t) \right| \mathcal{F}_t \right].
\end{eqnarray*}
\smallskip

b) If $g_{\rho}$ is positively homogeneous (hence $\rho_t$ is coherent), then
\begin{eqnarray}
\tilde{L}^Y (T;t) &=& \int_0^1 \frac{\mathcal{E} (\partial g(s ,Z_s^{\gamma Y}) \cdot B)(T)}{\mathcal{E} (\partial g(s,Z_s^{\gamma Y}) \cdot B)(t)} \, d\gamma \notag \\
&=& \int_0^1 \exp \left\{-\frac 12 \int_t^T \Vert \partial g(u,Z_u^{\gamma Y})\Vert ^2 \, d u - \int_t^T \partial g(u,Z_u^{\gamma Y}) \, dB_u \right\} \, d\gamma \notag \\
&=& \int_0^1 \exp \left\{-\frac 12 \int_t^T \Vert \partial g(u,\gamma Z_u^{ Y})\Vert ^2 \, d u - \int_t^T \partial g(u,\gamma Z_u^Y) \, dB_u \right\} \, d\gamma \label{eq: TC1}\\
&=& \int_0^1 \exp \left\{-\frac 12 \int_t^T \Vert \partial g(u, Z_u^{ Y})\Vert ^2 \, d u - \int_t^T \partial g(u, Z_u^Y) \, dB_u \right\} \, d\gamma \label{eq: TC2}\\
&=&\exp \left\{-\frac 12 \int_t^T \Vert \partial g(u, Z_u^{ Y})\Vert ^2 \, d u - \int_t^T \partial g(u, Z_u^Y) \, dB_u \right\}, \label{eq: TC3}
\end{eqnarray}
where~\eqref{eq: TC1} is due to $Z_s^{\gamma Y}=Z^{\gamma}_s$ and~\eqref{eq: TC2} to $\partial g(u,\gamma Z_u^{ Y})=\partial g(u, Z_u^{ Y})$, both satisfied for any $\gamma \in [0,1]$ because of positive homogeneity of $g(u,z)$ in $z$. By~\eqref{eq: TC3}, it then follows that $\tilde{L}^Y (T;s)= \tilde{L}^Y (T;t) \cdot \tilde{L}^Y (t;s)$ for any $s\leq t \leq T$, hence time-consistency of $\Lambda_t^{AS}$.
\smallskip

c)
No-undercut follows from
\begin{eqnarray*}
\Lambda_t^{c-AS}(X;Y)&=& \int_0^1 \left[ E_{Q_t^{\gamma Y}}[-X \vert \mathcal{F}_t]-c_{t}(Q_t^{\gamma Y}) \right] d\gamma \\
&\leq & \int_0^1 \rho_t(X) d\gamma =\rho_t(X)
\end{eqnarray*}
for any $X,Y \in L^{\infty}(\mathcal{F}_T), t \in [0,T]$.
Finally, from a) it follows that
\begin{equation*}
\Lambda_t^{c-AS}(X;X)= \Lambda_t^{AS}(X;X)-\int_0^1 c_{t}(Q_t^{\gamma X})  d\gamma = \rho_t(X) -\int_0^1 c_{t}(Q_t^{\gamma X})  d\gamma
\end{equation*}
in general does not coincide with $\rho_t(X)$, implying that $\Lambda_t^{c-AS}$ is not a dynamic CAR but only an audacious dynamic CAR.
\end{proof}
\smallskip

We conclude this section with an example of different dynamic CARs associated to the same risk measure. Some related remarks and discussions will follow.
\color{black}
It is worth emphasizing that, although the example below deals with differentiable $g_{\Lambda}$ and $g_{\rho}$ (hence with Gateaux differentiable risk measures), there are several examples of CARs induced by $g$-expectations and that are only subdifferentiable, e.g., the one corresponding to $g_{\rho}(s,z)=c \vert z \vert$ and $g_{\Lambda}(s,z,z^y)=c (\vert z^y \vert + \vert z^y -z \vert)$. Although the case below goes beyond the Lipschitz case, we have decided to focus on it mainly because it corresponds to one among the few examples of BSDEs where one has an explicit solution (see, e.g.,~\cite{barrieu-elkaroui}), but also because the entropic risk measure is quite popular and widely used also due to its relation to utility theory.
\normalcolor

\begin{example}[Entropic risk measures with different dynamic CARS] \label{ex: entropic- different CARS}

Let $\rho_t$ be a dynamic entropic risk measure, i.e., $\rho_t(X)=\lambda \ln \left(E_P [\exp(-\frac{X}{\lambda}) | \mathcal{F}_t] \right)$ where $\lambda >0$ denotes the risk aversion parameter. \color{black}It is well known that the dynamic entropic risk measure solves a BSDE with driver ${\color{black}g_{\rho}}(s,z)=\frac{1}{2 \lambda} \Vert z \Vert ^2$ (see, e.g.,~\cite{barrieu-elkaroui}).\normalcolor
\smallskip

In the following, we consider and compare different dynamic CARs associated to the entropic risk measure above by considering the approach presented before even if $g_{\color{black} \rho}$ is not Lipschitz in $z$.\medskip

\emph{Dynamic gradient CAR.} Since $g_{\color{black}\rho}$ is differentiable in $z$, the dynamic gradient CAR is well-defined and corresponds to
\begin{equation*}
\Lambda_t^{grad}(X;Y)= E_P \left[\left. {\color{black}-X \frac{e^{-\frac{Y}{\lambda}}}{E_P \left[\left.e^{-\frac{Y}{\lambda}} \right| \mathcal{F}_t \right]}} \right| \mathcal{F}_t \right],
\end{equation*}
(see \cite{kromer-overbeck-bsde} and \cite{mabitsela-etal}), solving a BSDE with driver $g_{\Lambda}(s,z,z^y)=\nabla g_{\rho}(s,z^y) \cdot z= \frac{1}{\lambda} z \cdot z^y$.
\smallskip

Consider now the following \emph{two alternative drivers} $g_{\Lambda}$:
\begin{equation*}
g^1_{\Lambda}(s,z,z^y)=c(z-z^y) + \frac{1}{2 \lambda} \Vert z^y \Vert^2; \quad g^2_{\Lambda}(s,z,z^y)=\frac{1}{2 \tilde{\lambda}}\Vert z-z^y \Vert ^2 + \frac{1}{2 \lambda} \Vert z^y \Vert^2
\end{equation*}
for some $c, \tilde{\lambda}>0$ where $\tilde{\lambda}$ can be seen as a second risk aversion parameter.
Note that $g^1_{\Lambda}(s,z,z)=g^2_{\Lambda}(s,z,z)=g_{\rho}(s,z)${\color{black}, $g^1_{\Lambda}$ is a Lipschitz driver (in $z$), while $g^2_{\Lambda}$ is a quadratic driver (in $z$) of a similar form of $g_{\rho}$, hence guaranteeing existence and uniqueness of the solution}.

Starting with $g^1_{\Lambda}$,
\begin{eqnarray*}
&&\Lambda_t^{1}(X;Y)-\rho_t(Y) \\
&=& -(X-Y)+ \int_t ^T \left[c(Z_s^{X,Y}-Z_s^Y) + \frac{1}{2 \lambda}  \Vert Z_s^Y \Vert ^2 \right] \, ds - \int_t^T Z_s^{X,Y} \, dB_s +\\
&& - \int_t ^T  \frac{1}{2 \lambda}  \Vert Z_s^Y \Vert ^2  \, ds + \int_t^T Z_s^Y \, dB_s \\
&=& -(X-Y)+ \int_t ^T \left[c(Z_s^{X,Y}-Z_s^Y)\right]  \, ds - \int_t^T \left( Z_s^{X,Y} - Z_s^Y \right) \, dB_s,
\end{eqnarray*}
hence
\begin{eqnarray*}
\Lambda_t^{1}(X;Y)&=&\rho_t(Y) - E_{Q_{t,c}} [-(Y-X) | \mathcal{F}_t] \\
&=&\lambda \ln \left(E_P \left[ \left. \exp\left(-\frac{Y}{\lambda}\right) \right| \mathcal{F}_t \right] \right)- E_{Q_{t,c}} [-(Y-X) | \mathcal{F}_t],
\end{eqnarray*}
with $\frac{dQ_{t,c}}{dP}= \mathcal{E} (c \cdot B)= \exp \left\{-\frac 12 c ^2 t - c \, B_t \right\}$.

Concerning $g^2_{\Lambda}$, instead,
\begin{eqnarray*}
&&\Lambda_t^{2}(X;Y)-\rho_t(Y) \\
&=& -(X-Y)+ \int_t ^T \left[\frac{1}{2 \tilde{\lambda}}\Vert Z_s^{X,Y}-Z_s^Y \Vert ^2 + \frac{1}{2 \lambda} \Vert Z_s^Y \Vert ^2 \right] \, ds - \int_t^T Z_s^{X,Y} \, dB_s +\\
&& - \int_t ^T  \frac{1}{2 \lambda} \Vert Z_s^Y \Vert^2  \, ds + \int_t^T Z_s^Y \, dB_s \\
&=& -(X-Y)+ \int_t ^T \left[\frac{1}{2 \tilde{\lambda}} \Vert Z_s^{X,Y}-Z_s^Y \Vert^2\right]  \, ds - \int_t^T \left( Z_s^{X,Y} - Z_s^Y \right) \, dB_s.
\end{eqnarray*}
Since given a constant $a>0$, a final condition $\xi$ and a fixed process $(b_t)_{t \in [0,T]}$ the BSDE
\begin{equation*}
Y_t= -\xi+ \int_t^T \frac{1}{2a} \Vert Z_s - b_s \Vert ^2 \, ds -\int_t ^T (Z_s-b_s) \, dB_s
\end{equation*}
admits a unique solution with first component $Y_t=a \ln \left(E_P [\exp(-\frac{\xi}{a}) | \mathcal{F}_t] \right)$, $\Lambda_t^{2}$ is then given by
\begin{eqnarray*}
\Lambda_t^{2}(X;Y)&=&\rho_t(Y)+ \tilde{\lambda} \ln \left(E_P \left[\left. \exp\left(-\frac{X-Y}{\tilde{\lambda}}\right) \right| \mathcal{F}_t \right] \right) \\
&=& \ln \left[ \left(E_P \left[\left. \exp\left(-\frac{Y}{\lambda}\right) \right| \mathcal{F}_t \right] \right)^{\lambda} \left(E_P \left[\left. \exp\left(-\frac{X-Y}{\tilde{\lambda}}\right) \right| \mathcal{F}_t \right]\right)^{\tilde{\lambda}} \right].
\end{eqnarray*}

As seen above, we can associate several dynamic CARs to the entropic risk measure (as well as to any risk measure). In particular, the choice of $g_{\Lambda}$ and of the corresponding dynamic CAR may reflect the preference (or the risk aversion) of the agent/intermediary/investor. This financial interpretation is particularly evident for $g_{\Lambda}^2$ depending on the two risk aversion parameters $\lambda, \tilde{\lambda}$ where $\lambda$ can be seen as the risk aversion at the level of the whole portfolio $Y$ while $\tilde{\lambda}$ at the level of the rest of the portfolio with the exclusion of $X$. In particular, $g_{\Lambda}^2 $ incorporates the drivers of two different entropic terms with different risk aversion and, consequently, the corresponding CAR is somehow related to different entropic risk measures acting on $Y$ and on $Y-X$.
\end{example}

Note that $g_{\Lambda}^1, g_{\Lambda}^2$ of the previous example are particular cases of $g^f_{\Lambda}(s,z,z^y)=f(s,z^y,z-z^y) + g_{\rho} (s,z^y)$ with $f(s,z^y,0)=0$. In such a case,
\begin{eqnarray*}
\Lambda_t^{f}(X;Y) &=& -X+ \int_t ^T \left[f(s,Z_s^Y, Z_s-Z_s^Y) + g_{\rho}(s,Z_s^Y) \right] \, ds - \int_t^T  Z_s  \, dB_s \\
&=& -(X-Y)+ \int_t ^T f(s,Z_s^Y, Z_s-Z_s^Y)  \, ds - \int_t^T \left[ Z_s - Z_s^Y \right] \, dB_s +\\
&& -Y+ \int_t ^T g_{\rho}(s,Z_s^Y)  \, ds - \int_t^T Z_s^Y \, dB_s \\
&=& \rho_t(Y) -(X-Y)+ \int_t ^T (Z_s-Z_s^Y) \frac{f(s,Z_s^Y, Z_s-Z_s^Y)}{Z_s-Z_s^Y}  \, ds - \int_t^T \left( Z_s - Z_s^Y \right) \, dB_s \\
&=& \rho_t(Y) - E_{Q_{t,f}} [-(Y-X) | \mathcal{F}_t]
\end{eqnarray*}
with $\frac{dQ_{t,f}}{dP}= \mathcal{E} (\frac{\Delta f}{\Delta z} \cdot B)$ and $\frac{\Delta f}{\Delta z}=\frac{f(s,Z_s^Y, Z_s-Z_s^Y)}{Z_s-Z_s^Y}$. In other words, when the driver of the CAR is given by that of the risk measure plus an additional term depending on the difference $z-z^y$, the dynamic CAR can be obtained by the riskiness of the whole portfolio with a correction term depending on $Y-X$, that is, on the rest of the portfolio with the exclusion of $X$. Roughly speaking, this approach is somehow similar to that of marginal contributions where the contribution of $X$ on the whole portfolio $Y$ is taken into account.
\smallskip

The main aim of this paper has been to provide an axiomatic study of dynamic CARs and to introduce a general formulation of them in a BSDE-framework, both under the subdifferentiability assumption, weaker than Gateaux differentiability. The following remark illustrates why the subdifferentiability assumption is rather weak for $g_{\rho}$ and what is the impact of subdiffertiability of $g_{\Lambda}$ on $\Lambda_t$ and on $\rho_t$.

\begin{remark} \label{rem: subdif g}
\color{black}
Although $g_{\rho}$ is assumed to be convex in $z \in \mathbb{R}^n$ (because of convexity of $\rho_t$), $g_{\Lambda}(s,z,z^y)$ is not necessarily convex in $z$. Consequently, while $\partial g_{\rho}(t,z) \neq \emptyset$ on the relative interior of the domain of $g_{\rho}$ (see Rockafellar~\cite{rockafellar}), $g_{\Lambda}(s,z,z^y)$ is not necessarily subdifferentiable in $z$.

We claim that subdifferentiability of $g_{\Lambda}$ in $z$, however, implies subdifferentiability of $g_{\rho}$, $\rho_t$ and $\Lambda_t$, where that of $\Lambda_t$ should be understood as
\begin{multline*}
\partial \Lambda_t(X;Y)= \left\{Q_t \in \mathcal{Q}_{t}: \Lambda_t(H;Y) \geq \right.\\
\left. \Lambda_t(X;Y) + E_{Q_t}[-(H-X) | \mathcal{F}_t] \mbox{ for all } H \in L^\infty(\mathcal{F}_T)\right\} \neq \emptyset
\end{multline*}
for any $Y \in L^{\infty}(\mathcal{F}_T)$.

In fact, subdifferentiability of $g_{\rho}$ follows immediately since $g_{\rho}(s,z)= g_{\Lambda}(s,z,z)$.

Subdifferentiability of $g_{\Lambda}$, then, implies that
\begin{eqnarray*}
&& \Lambda_t(H;Y)-\Lambda_t(X;Y) \\
&=& -(H-X)+ \int_t ^T \left[g_{\lambda}(s,Z_s^{H,Y}, Z_s^Y)-g_{\lambda}(s,Z_s^{X,Y}, Z_s^Y) \right] \, ds - \int_t^T \left[ Z^{H,Y}_s - Z_s^{X,Y} \right] \, dB_s\\
&\geq & -(H-X)+ \int_t ^T \partial g_{\lambda}(s,Z_s^{X,Y}, Z_s^Y) \left[Z_s^{H,Y} - Z_s^{X,Y} \right] \, ds - \int_t^T \left[ Z^{H,Y}_s - Z_s^{X,Y} \right] \, dB_s \\
&=& E_{Q_{t}^{Y,X}} [-(H-X) | \mathcal{F}_t],
\end{eqnarray*}
for any $H, X, Y \in L^{\infty}(\mathcal{F}_T)$, where $E\left[\left.\frac{dQ_t^{Y,X}}{dP} \right| \mathcal{F}_t \right]= \mathcal{E} (\partial g_{\lambda}(s,Z_s^{X,Y}, Z_s^Y) \cdot B)$. It then follows that $Q_t^{Y,X} \in \partial \Lambda_t(X;Y)$, so both $\Lambda_t(X;Y)$ and $\rho_t(Y)$ are subdifferentiable.
\normalcolor
\end{remark}

\section{Conclusions} \label{sec: conclus}

In this paper, we have introduced a general axiomatic approach to
dynamic capital allocations as well as an approach suitable for
risk measures induced by $g$-expectations, by going beyond the gradient approach and by weakening the Gateaux differentiability condition.
%

Furthermore, for risk measures induced by $g$-expectations we have seen that it is possible to associate several dynamic capital allocations induced by a $g_{\Lambda}$-expectation, where the choice of the driver $g_{\Lambda}$ and of the corresponding dynamic
CAR may reflect the preferences of the intermediary/investor.
Vice versa, instead, given a dynamic family of CARs induced by a
$g$-expectation, the associated dynamic risk measure is
uniquely determined.
\bigskip
\bigskip

\color{black}
\textit{Acknowledgments.} We wish to thank two anonymous Referees for their comments that contributed to improve the paper. We are also grateful to Alessandro Doldi, Marco Frittelli, and Marco Zullino for their suggestions and discussions.
\normalcolor

\end{document}